\documentclass[12pt,a4paper]{article}

\usepackage{amsthm}
\usepackage{tikz,color}

\newtheorem{theorem}{Theorem}
\newtheorem{thm}{Theorem}[section]

\newtheorem{claim}[thm]{Claim}

\theoremstyle{remark}
\newtheorem{remark}{Remark}[section]

\usepackage{amsmath,amssymb,amsfonts,latexsym,epsfig,graphicx,longtable,wrapfig}

\title{Maximum chordal subgraphs of random graphs}

\author{Michael Krivelevich\thanks{School of Mathematical Sciences, Tel Aviv University, Tel Aviv 6997801, Israel. 
\newline
Email: krivelev@tauex.tau.ac.il. Research supported in part by USA-Israel BSF grant 2018267.}, Maksim Zhukovskii\thanks{The University of Sheffield, Department of Computer Science, Sheffield, UK.\newline Email: m.zhukovskii@sheffield.ac.uk.}}

\date{}

\begin{document}

\maketitle

\begin{abstract}
We find asymptotics of the maximum size of a chordal subgraph in a binomial random graph $G(n,p)$, for $p=\mathrm{const}$ and $p=n^{-\alpha+o(1)}$.
\end{abstract}

\section{Introduction}

A {\it chordal graph} is a graph with no induced cycles of length at least 4. Chordal graphs are one of the most studied classes of perfect graphs and graphs in general due to their beautiful characterisations, useful and diverse properties, and various applications: chordal graphs arise in constraint programming, relational databases, Bayesian networks for probabilistic reasoning, register allocation, etc. In particular, chordal completions of graphs are used to characterise some graph classes, to define the treewidth, and are related to important computational problems (see~\cite{Computers}). Structural properties of this family of graphs  help in solving hard problems (such as proper colouring, maximum clique and independent set) efficiently. Chordal graphs are also used for reconstructing evolutionary trees~\cite{Buneman} and are applied for semidefinite optimisation~\cite{VA_survey}. We refer the reader to~\cite{Golumbic_book,BGW_book} for comprehensive surveys on chordal graphs.

Given the prominence of chordal graphs, it is natural to expect extremal problems involving them to be studied. For example, in 1985, Erd\H{o}s and Laskar~\cite{EL} posed the question about the maximum integer $\ell(n,m)$ such that every graph on $[n]:=\{1,\ldots,n\}$ with $m$ edges contains a chordal subgraph with at least $\ell(n,m)$ edges. The question was answered by Gishboliner and Sudakov in~\cite{GS} --- they determined the value of $f(n,m)$ for all $m$ up to a $O(\sqrt{n})$ additive term, and found the exact value for all $m\leq\frac{n^2}{3}+1$. In particular, if $m<(1-\varepsilon){n\choose 2}$, then $f(n,m)<Cn$ for some $C=C(\varepsilon)>0$, and this is not surprising since a $K_s$-free graph on $n$ vertices does not contain a chordal subgraph with more than $(s-2)n$ edges, and, for $s$ large enough, there are $K_s$-free graphs with $m$ edges. 

In the last decades, there has been a great interest in investigating (or even transferring results related to) extremal combinatorial questions in random graph settings. Most attention was given to the Tur\'{a}n's problem (see, e.g.,~\cite{BPS,CG,MK,HS,KKM,Schacht}), and, more generally, to determining the maximum number of edges in a subgraph of random graph that belongs to a given family of graphs (see,~e.g.,~\cite{AKSamotij,BF,CMS,DMS,GL}). 

Here we consider an extremal question about chordal subgraphs in a random setting. Gishboliner in his talk~\cite{Gish} asked the following average-case question: what is the size $X_n$ of a largest chordal subgraph in the binomial random graph $G(n,p=\mathrm{const})$? In the paper we answer this question asymptotically. We also find asymptotics of the maximum size of a chordal subgraph in $G(n,n^{-\alpha+o(1)})$ for all $\alpha\neq\frac{1+k}{1+2k}$, $k\in\mathbb{Z}_{\geq 0}$.\\

We will make use of an equivalent definition of chordal graphs. In particular, it can be used to get a simple upper bound on the maximum size of a chordal subgraph in a random graph, as we show below. Let us recall that a {\it perfect elimination ordering} $v_1\prec\ldots\prec v_n$ of the vertices of a graph $H$ satisfies the following requirement: for every $i\in[n]$, the set of {\it outgoing neighbours} (i.e. the neighbours of $v_i$ among $v_1,\ldots,v_{i-1}$) induces a clique in $H$. It is easy to see that the existence of a perfect elimination ordering implies that $H$ is chordal. The opposite is also true; thus, a graph is chordal if and only if it admits a perfect elimination ordering of its vertices~\cite{chordal},~\cite[Chapter 5]{West}. 

Assuming that a chordal graph $H$ on $[n]$ has $\ell$ edges, we get that for any perfect elimination ordering, a certain vertex has at least $\ell/n$ outgoing neighbours, thus $H$ has a clique with more than $\ell/n$ vertices. Therefore, bounding the clique number of a graph provides an immediate upper bound on the maximum size of its chordal subgraph. Let $p=\mathrm{const}\in(0,1)$. Let us recall that whp\footnote{With high probability, that is, with probability tending to 1 as $n\to\infty$.} the clique number of $G_n\sim G(n,p)$ equals $(2-o(1))\log_{1/p}n$, and one can cover $n(1-o(1))$ vertices by cliques of about this size~\cite[Theorem 7.1, Lemma 7.13]{Janson}. On the one hand, it immediately implies that whp $X_n\leq 2n\log_{1/p}n$. On the other hand, since a disjoint union of cliques is chordal, we get $X_n\geq (1-o(1))n\log_{1/p}n$ whp. We shall prove that neither of these bounds is asymptotically tight.\\

 Let $\gamma$ be the unique solution in $(\max\{1,2p\},2)$ of the equation
\begin{equation}
\gamma\ln\frac{2}{\gamma}+(2-\gamma)\ln\frac{2}{2-\gamma}=(2-\gamma)\ln\frac{1}{1-p}-(1-\gamma)\ln\frac{1}{p}.
\label{eq:gamma_def}
\end{equation}
Note that a solution in $(\max\{1,2p\},2)$ of~\eqref{eq:gamma_def} is indeed unique: denote the difference between the left-hand side and the right-hand side in~\eqref{eq:gamma_def} by $g(\gamma)$. Since 
  $g'(\gamma)=\ln\frac{(2-\gamma)p}{\gamma(1-p)}$, thus $g$ decreases on $(2p,2)$. On the other hand, the values of this difference 
at $2p$ and $1$ equal $g(2p)=\ln\frac{1}{p}>0$ and $g(1)=\ln(4(1-p))$, which is positive whenever $p<\frac{3}{4}$. It remains to observe that $\max\{1,2p\}=2p$ when $p\geq\frac{3}{4}$ and that $g(2^-)=-\ln\frac{1}{p}<0$.



\begin{theorem}
Let $p=\mathrm{const}\in(0,1)$. Then whp $\left|X_n-\gamma n\log_{1/p} n\right|\leq 10n\log_{1/p}\log n$.
\label{th:main}
\end{theorem}

Note that $\gamma>1$, so a union of disjoint cliques of size $(2-o(1))\log_{1/p}n$ is indeed not an asymptotically optimal choice of a chordal subgraph. In particular, when $p=1/2$, then $\gamma=1.7799...$ satisfies $\gamma\ln\frac{2}{\gamma}+(2-\gamma)\ln\frac{2}{2-\gamma}=\ln 2.$ We prove Theorem~\ref{th:main} in Section~\ref{sc:proof_th1}. Note that $\gamma$ is chosen so that $\mathbb{P}(\mathrm{Bin}(k,p)=s)=n^{-1+o(1)}$ for $k\sim 2\log_{1/p}n$ and $s\sim \gamma\log_{1/p}n$ --- see Claim~\ref{cl:main_technical_inequality}. This equality is crucial for both the upper and the lower bounds on $X_n$ that are proven in Sections~\ref{sc:proof_dense_lower}~and~\ref{sc:ub} respectively. In particular, it guarantees that, for an appropriately chosen $k$ and $s$ as above, whp every vertex of $[n]\setminus[n/\ln n]$ has $s$ neighbours in some $k$-clique which is entirely inside $[n/\ln n]$. This observation facilitates proving the lower bound, see Section~\ref{sc:proof_dense_lower} for details.\\

We then investigate the problem of tractability of searching for large chordal subgraphs in random graphs. In Section~\ref{sc:complexity} we show that there is a polynomial-time algorithm that finds in $G_n$ a chordal subgraph with $(1-o(1))n\log_{1/p}n$ edges whp and that a further improvement of this result is quite unlikely.\\ 

We also study maximum sizes of chordal subgraphs in sparse random graphs. Note that, when $p<n^{-\varepsilon}$ for some constant $\varepsilon>0$, whp $G_n\sim G(n,p)$ does not contain cliques of size $\lceil 2/\varepsilon\rceil +1$ implying that whp $G_n$ does not contain chordal graphs with at least $\lceil 2/\varepsilon\rceil n$ edges. We found asymptotics of the maximum size of a chordal subgraph of $G(n,n^{-\alpha+o(1)})$ for all positive constants $\alpha\neq\frac{1+k}{1+2k}$, $k\in\mathbb{Z}_{\geq 0}$. 

First of all, for the very sparse case $p\leq c/n$, for some constant $0<c<1$, recall that whp all (for $p=o(1/n)$) or nearly all (for $p=\Theta(1/n)$) edges of $G_n$ lie in tree components. Thus, by taking $F$ to be a maximal spanning forest of $G_n$, we conclude that in this regime whp $X_n=\binom{n}{2}p(1+o(1))$ (unless $p=\Theta(n^{-2})$ --- in that case the number of edges is not concentrated, and $X_n$ equals the number of edges, which is bounded in probability).


The following theorem establishes the limit in probability of $X_n/n$ when $p=n^{-\alpha+o(1)}$, for $\alpha\in(0,1]$ such that $\alpha\neq\frac{1+k}{1+2k}$, for all $k\in\mathbb{Z}_{\geq 0}$.

\begin{theorem}
Let $p=n^{-\alpha+o(1)}$, for a constant $\alpha>0$.
\begin{enumerate}
\item If $\alpha\in(1/2,1]$ and $\alpha\notin\left\{\frac{1+k}{1+2k}:\,k\in\mathbb{Z}_{\geq 0}\right\}$, 
  then $X_n/n\stackrel{\mathbb{P}}\to\frac{1+2k_{\alpha}}{1+k_{\alpha}}$,
where $k_{\alpha}$ is the largest integer such that $\alpha<\frac{1+k_{\alpha}}{1+2k_{\alpha}}$.
\item For any $\alpha\in(0,1/2]$, we have that  $X_n/n\stackrel{\mathbb{P}}\to 1/\alpha$.
\end{enumerate}
\label{th:main_2}
\end{theorem}

The proof of Theorem~\ref{th:main_2} appears in Section~\ref{sc:sparse}. In Section~\ref{sc:remarks} we pose several further questions.

\section{Dense random graphs: proof of Theorem~\ref{th:main}}
\label{sc:proof_th1}

\subsection{Preliminaries}

Let $p=\mathrm{const}\in(0,1)$,
$$
k_+=k_+(n,p)=\lceil 2\log_{1/p}n\rceil,
\quad
k_-=k_-(n,p)=\lceil  2\log_{1/p}n-9\log_{1/p}\log_{1/p} n\rceil. 
$$
We will use the following well-known bounds on the maximum size of a clique in $G_n$ (see~e.g.,~\cite[Theorem 7.1, Lemma 7.13]{Janson}).

\begin{claim}
Whp in $G_n$
\begin{enumerate}
\item there are no cliques of size $k_+$;
\item every set of vertices of size at least $\frac{n}{\ln^2 n}$ contains a clique of size $k_-$.
\end{enumerate}
\label{cl:cliques}
\end{claim}

Also, for technical reasons, we need to estimate the probability that a fixed vertex sends $s=(\gamma+o(1))\log_{1/p} n$ edges to a fixed set of size about $(2+o(1))\log_{1/p}n$. Note that, as follows from the claim below, $\gamma$ is defined in \eqref{eq:gamma_def} in such a way that $\mathbb{P}[\mathrm{Bin}(k,p)=s]=n^{-1+o(1)}$.

\begin{claim}
Let $x\neq\frac{1}{2(1-\log_{1/p}2/\gamma)}$ be an arbitrary (not necessarily positive) real number, 
$$
k=2\log_{1/p}n-x\log_{1/p}\log n +O(1),\quad s=\gamma\log_{1/p}n-x\log_{1/p}\log n +O(1)
$$
be integers. Then 
$$
 \mathbb{P}[\mathrm{Bin}(k,p)=s]=\frac{1}{n}\exp\biggl[\ln\ln n\left(x\left(1-\log_{1/p}\frac{2}{\gamma}\right)-\frac{1}{2}\right)(1+o(1))\biggr].
$$
\label{cl:main_technical_inequality}
\end{claim}
\begin{proof}
Set $\varepsilon=x\frac{\ln\ln n}{\ln n}$. Then
\begin{align*}
{k\choose s} & =\Theta\left(\frac{1}{\sqrt{\ln n}}\left(\frac{k^k}{s^s(k-s)^{k-s}}\right)\right)
 =\Theta\left(\frac{1}{\sqrt{\ln n}}\left(\frac{k}{s}\right)^s\left(\frac{k}{k-s}\right)^{k-s}\right)\notag\\
 &=
 \Theta\left(\frac{1}{\sqrt{\ln n}}\left(\frac{2-\varepsilon}{\gamma-\varepsilon}\right)^{(\gamma-\varepsilon)\log_{1/p} n}\left(\frac{2-\varepsilon}{2-\gamma}\right)^{(2-\gamma)\log_{1/p} n}\right)\notag\\
 &=\exp\left[\log_{1/p} n\left((\gamma-\varepsilon)\ln\frac{2-\varepsilon}{\gamma-\varepsilon}+(2-\gamma)\ln\frac{2-\varepsilon}{2-\gamma}\right)-\frac{1}{2}\ln\ln n(1+o(1)))\right].
 \end{align*}
 Since
\begin{multline*}
(\gamma-\varepsilon)\ln\frac{2-\varepsilon}{\gamma-\varepsilon}+(2-\gamma)\ln\frac{2-\varepsilon}{2-\gamma}=\\
(\gamma-\varepsilon)\ln\frac{2}{\gamma}+(\gamma-\varepsilon)\ln\left(1-\frac{\varepsilon}{2}\right)-(\gamma-\varepsilon)\ln\left(1-\frac{\varepsilon}{\gamma}\right)+(2-\gamma)\ln\frac{2}{2-\gamma}+(2-\gamma)\ln\left(1-\frac{\varepsilon}{2}\right)\\
=\gamma\ln\frac{2}{\gamma}+(2-\gamma)\ln\frac{2}{2-\gamma}-\varepsilon\ln\frac{2}{\gamma}+O(\varepsilon^2),
\end{multline*}
we get that 
 \begin{align*}
  {k\choose s} &=\exp\left[\log_{1/p} n\left(\gamma\ln\frac{2}{\gamma}+(2-\gamma)\ln\frac{2}{2-\gamma}-\varepsilon\ln\frac{2}{\gamma}\right)-\frac{1}{2}\ln\ln n(1+o(1))\right]\notag\\
    &\stackrel{\eqref{eq:gamma_def}}=\exp\left[\log_{1/p} n\left((2-\gamma)\ln\frac{1}{1-p}-(1-\gamma)\ln\frac{1}{p}-\varepsilon\ln\frac{2}{\gamma}\right)-\frac{1}{2}\ln\ln n(1+o(1))\right].\notag\\
\end{align*}
Thus,
\begin{align*}
\mathbb{P}[\mathrm{Bin}(k,p)=s]&= {k\choose s}p^s(1-p)^{k-s}\\
 &=\exp\biggl[\ln n\left((2-\gamma)\log_{1/p}\frac{1}{1-p}+\gamma-\varepsilon\log_{1/p}\frac{2}{\gamma}-1\right)\\
 &\quad\quad\quad\quad\quad
 -s\ln\frac{1}{p}-(k-s)\ln\frac{1}{1-p}-\left(\frac{1}{2}+o(1)\right)\ln\ln n\biggr]\\
 &=\frac{1}{n}\exp\biggl[\varepsilon\left(1-\log_{1/p}\frac{2}{\gamma}\right)\ln n-\frac{1}{2}\ln\ln n(1+o(1))\biggr],
\end{align*}
as needed.
\end{proof}

\begin{remark}
In order to justify the choice of the constant factor in the second order term in the definition of $k_-$, let us observe that \eqref{eq:gamma_def} implies
$$
\log_{1/p}\frac{2}{\gamma}=1-\frac{1}{\gamma}-\frac{2-\gamma}{\gamma}\log_{1/p}\frac{2(1-p)}{2-\gamma}<\frac{1}{2}
$$
since $\gamma<2$ and $\frac{2(1-p)}{2-\gamma}>1$. Thus, we get that there exists $\varepsilon=\varepsilon(p)>0$ such that
\begin{align*}
\mathbb{P}[\mathrm{Bin}(k,p)=s]\geq\frac{1}{n}\exp[\ln\ln n(x+\varepsilon-1-o(1))/2],\quad\text{if }x>0,\\
\mathbb{P}[\mathrm{Bin}(k,p)=s]\leq\frac{1}{n}\exp[\ln\ln n(x-\varepsilon-1-o(1))/2],\quad\text{if }x<0.
\end{align*}
\label{rk:choice_of_C}
\end{remark}

\subsection{Lower bound}
\label{sc:proof_dense_lower}

We let $n'= \lfloor n/\ln n\rfloor$ and $k=k_-(n',p)$. Note that $k=(2-o(1))\log_{1/p}n$. We then divide $[n]$ into two parts $V\sqcup U$, where $V$ has size $n'$. Expose first the edges of $G_n$ spanned by $V$. Due to Claim~\ref{cl:cliques}, whp there exists a set $V_0\subset V$ of size at most $n/\ln^2 n$ such that $G(n,p)[V\setminus V_0]$ can be partitioned into cliques of size $k$. Let $K_1,\ldots,K_m$ be such cliques, where 
$$
m=(1+o(1))\frac{n}{2\ln n\log_{1/p}n}.
$$

Let us now expose edges of $G_n$ between $V$ and $U$. Let $s$ be the maximum integer such that 
$$
\mathbb{P}[\mathrm{Bin}(k,p)\geq s]\geq \mathbb{P}[\mathrm{Bin}(k,p)=s]\geq\frac{\ln^4 n}{n}.
$$
Then a vertex $u\in U$ has at least $s$ neighbours in some $K_j$, $j\in[m]$, with probability at least 
$$
1-\left(1-\frac{\ln^4 n}{n}\right)^m=1- e^{-\Omega(\ln^2n)}=1-o(1/n).
$$
By the union bound, whp, for every vertex $u\in U$, there is $j\in[m]$ such that $u$ has at least $s$ neighbours in $K_j$.

Consider a subgraph $H$ of $G_n$ obtained from the disjoint union of cliques $K_1,\ldots,K_m$ with vertices from $U$, each sending $s$ edges to exactly one of the cliques. This graph has 
$$
m{k\choose 2}+s|U|\geq\gamma n\log_{1/p}n-9n\log_{1/p}\log n+O(n)
$$ 
edges due to Claim~\ref{cl:main_technical_inequality}. Indeed, letting $x=9$, we get from Claim~\ref{cl:main_technical_inequality}, Remark~\ref{rk:choice_of_C}, and the definition of $s$ that $s>\gamma\log_{1/p}n-9\log_{1/p}\log n +O(1)$.

It remains to prove that $H$ is chordal. Consider any ordering of vertices of $H$, where each vertex of $V$ precedes any vertex of $U$. Obviously, this is a perfect elimination ordering implying that $H$ is chordal.

\subsection{Upper bound}
\label{sc:ub}


Let $k=k_+(n,p)$. Letting 
$$
s=\gamma\log_{1/p}n+3\log_{1/p}\log n +O(1)
$$
to be an integer, we get that 
\begin{equation}
\mathbb{P}[\mathrm{Bin}(k,p)=s]\leq\mathbb{P}[\mathrm{Bin}(\lfloor k+3\log_{1/p}\log n\rfloor,p)=s]\leq\frac{1}{n}\exp[-2\ln\ln n].
\label{eq:bin_upper}
\end{equation}
The first inequality holds true since $\gamma>2p$ and thus $\mathbb{P}[\mathrm{Bin}(y,p)=s]$ increases in $y$ on $[s,k+\delta(k)]$ for any choice of $\delta(k)=o(k)$: indeed 
$$
\frac{\mathbb{P}[\mathrm{Bin}(y+1,p)=s]}{\mathbb{P}[\mathrm{Bin}(y,p)=s]}=\frac{(y+1)(1-p)}{y+1-s}>
\frac{k(1+o(1))(1-p)}{k(1+o(1))-s}=\frac{2(1-p)+o(1)}{2-\gamma}>1.
$$
The second inequality in \eqref{eq:bin_upper} holds true due to Claim~\ref{cl:main_technical_inequality} and Remark~\ref{rk:choice_of_C}. Let 
$$
\ell=sn=\gamma n\log_{1/p}n+3n\log_{1/p}\log n +O(n).
$$
We will prove that whp the maximum number of edges in a chordal subgraph of $G_n$ is less than $\ell$, and this immediately implies the desired upper bound.\\ 






Let $H$ be a chordal graph, and assume we are given a perfect elimination ordering $v_1\prec\ldots\prec v_n$ of the vertices of $H$. For every vertex $v_i$, let $d(v_i)$ be the {\it outdegree} of $v_i$, i.e. the number of outgoing neighbours of $v_i$, and let $K_{v_i}$ be the clique induced by $v_i$ and its outgoing neighbours. For every $i$ such that $d(v_i)>0$, let $\nu(v_i)$ be the last outgoing neighbour of $v_i$ in the given perfect elimination ordering.
   Consider a graph $T_{\prec}(H)$ on the vertex set $V(H)$ consisting of all edges $\{v_i,\nu(v_i)\}$. Note that $T_{\prec}(H)$ is 1-degenerate, thus it is a forest. 

We further assume that $H$ is connected. In this case, for any perfect elimination ordering $\prec$, we have that $T=T_{\prec}(H)$ is a tree. Indeed, it is sufficient to show that, for every pair of vertices $v_j\prec v_i$ forming an edge of $H$, there is a path in $T$ from $v_i$ to $v_j$. If $\nu(v_i)=v_j$, then $\{v_j,v_i\}$ is an edge of $T$ itself. Otherwise, $v_j\prec\nu(v_i)\prec v_i$, $\{v_i,\nu(v_i)\}$ is an edge of both $T$ and $H$, and thus $\{v_j,\nu(v_i)\}$ is an edge of $H$.  By induction, we eventually will get a path from $v_i$ to $v_j$ in $T$. Let us call $T$ a {\it perfect elimination tree of $H$}, and let $v_1$ be the root of $T$. Note that any layer-preserving ordering of the vertices of $T$ (i.e. vertices that are further from the root in $T$ occur later in this ordering) is a perfect elimination ordering of $H$.

For any rooted tree $T$ on $[n]$ there is at least one connected chordal graph $H$ such that $T$ is its perfect elimination tree (actually we may take $H=T$). 
 In order to recover an $H$ (uniquely) from $T$, we also equip $T$ with additional data: assume that $v_1\prec\ldots\prec v_n$ is a layer-preserving ordering of the vertices of $T$ and assign to each vertex $v_i$, $i\geq 2$, a vector $\mathbf{e}_i\in\{0,1\}^{k_i}$ that encodes the outgoing neighbourhood of the vertex $i$ in $H$, where $k_i=|K_{\nu(v_i)}|-1$ (see Fig.~1). Thus, in $H$, the vertex $i$ is adjacent to a vertex $j<\nu(v_i)$ if and only if $j\in K_{\nu(v_i)}$ and the respective coordinate of $\mathbf{e}_i$ equals 1. 
 Note that $k_2=0$, and, for every $i\geq 3$, vectors $\mathbf{e}_j$, $j\leq i-1$, define $k_i$ uniquely.\\
 
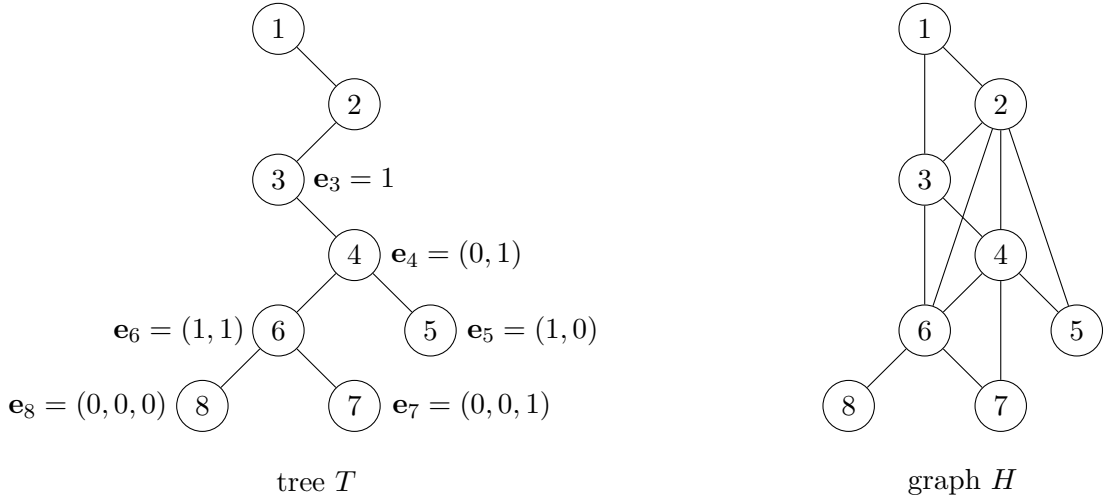
\begin{figure}
\begin{centering}
\begin{tikzpicture}[main/.style = {node distance={15mm}, draw, circle}] \small
\label{fig:1}
\node[main] at (-1, 2)  (a)  {1};

\node[main] at (0, 1)   (b) {2};

\node[main] at (-1, 0)   (c) {3};

\node[main] at (0, -1)  (d) {4};

\node[main] at (1, -2)  (e) {5};

\node[main] at (-1, -2)  (f) {6};

\node[main] at (0, -3)  (g) {7};

\node[main] at (-2, -3)  (h) {8};



\draw[black] (a) -- (b);
\draw[black] (b) -- (c);
\draw[black] (c) -- (d);
\draw[black] (d) -- (e);
\draw[black] (d) -- (f);
\draw[black] (f) -- (g);
\draw[black] (f) -- (h);

\node[main] at (7.5, 2)  (1)  {1};

\node[main] at (8.5, 1)   (2) {2};

\node[main] at (7.5, 0)   (3) {3};

\node[main] at (8.5, -1)  (4) {4};

\node[main] at (9.5, -2)  (5) {5};

\node[main] at (7.5, -2)  (6) {6};

\node[main] at (8.5, -3)  (7) {7};

\node[main] at (6.5, -3)  (8) {8};


\node[] at (0, 0) (333) {$\mathbf{e}_3=1$}; 
\node[] at (1.35, -1) (444) {$\mathbf{e}_4=(0,1)$}; 
\node[] at (2.35, -2) (555) {$\mathbf{e}_5=(1,0)$}; 
\node[] at (-2.3, -2) (666) {$\mathbf{e}_6=(1,1)$}; 
\node[] at (1.55, -3) (777) {$\mathbf{e}_7=(0,0,1)$}; 
\node[] at (-3.5, -3) (888) {$\mathbf{e}_8=(0,0,0)$};

\draw[black] (1) -- (2);
\draw[black] (2) -- (3);
\draw[black] (1) -- (3);
\draw[black] (3) -- (4);
\draw[black] (2) -- (4);
\draw[black] (4) -- (5);
\draw[black] (2) -- (5);
\draw[black] (4) -- (6);
\draw[black] (2) -- (6);
\draw[black] (3) -- (6);
\draw[black] (6) -- (7);
\draw[black] (4) -- (7);
\draw[black] (6) -- (8);

\node[] at (-0.5, -4) (222) {tree $T$}; 
\node[] at (8, -4) (222) {graph $H$};

\end{tikzpicture} 
\caption{Reconstruction of $H$ from $T$.}
\end{centering}
\end{figure}

We can now estimate the expected number of rooted trees $T$ on $[n]$ such that 
 in $G_n$ there exists a {\it connected} chordal subgraph $H$ with the following properties:
\begin{itemize}
\item $H$ has at least $\ell$ edges,
\item $H$ does not have cliques of size $k$,
\item$T$ is a perfect elimination tree of $H$.
\end{itemize}
For our goal, it is sufficient to prove that this expectation approaches 0 due to Claim~\ref{cl:cliques}.1. In particular, since whp $G_n$ is connected, a chordal subgraph with the maximum number of edges is also connected (any disconnected chordal subgraph on $[n]$ can be supplemented by an edge between any pair of connected components), thus the connectivity assumption does not cause a loss of generality. There are $n^{n-1}$ ways to construct a rooted tree $T$. Take a rooted tree $T$ and consider any layer-preserving ordering $\prec$. Without loss of generality, we assume that this ordering is defined by the identity permutation on $[n]$. 
 Thus, the desired expectation is bounded from above by $\rho n^{n-1}$, where $\rho$ is the maximum (over $T$) probability that $G_n$ has a chordal subgraph $H$ with $T=T_{\prec}(H)$, at least $\ell$ edges, and without cliques of size $k$. Let us expose the edges of $G_n$ in the following order that $\prec$ induces on the pairs $u<v$ of vertices in $[n]$: $(u,v)<(u',v+1)$ for any $u'$, and $(u,v)<(u+1,v)$. For any $v\geq 2$, as soon as $G_n[\leq v-1]$ is exposed, the outdegree of $v$ to the clique $K_{\nu(v)}$ (note that $\nu(v)$ is defined by $T$) has binomial distribution with $k_{\nu}(v)+1$ trials, and $k_{\nu}(v)+1$ is uniquely defined by $G_n[\leq v-1]$. 
  Let us also recall that each $k_{\nu(v)}+1$ should be at most $k$. We then get that $d(2)+\ldots+d(n)$ is stochastically dominated by the sum of $n-1$ independent $\mathrm{Bin}(k,p)$ random variables implying
\begin{align*}
 \rho\leq\mathbb{P}[d(2)+\ldots+d(n)\geq\ell]\leq\mathbb{P}[\mathrm{Bin}(kn,p)\geq\ell]&\stackrel{(*)}\leq
 kn\mathbb{P}[\mathrm{Bin}(kn,p)=\ell]\\
 &=  kn{kn\choose\ell}p^{\ell}(1-p)^{kn-\ell}.
\end{align*}
The inequality (*) holds since $\ell>(1+\varepsilon)knp$ for a sufficiently small constant $\varepsilon>0$ (that follows immediately from $p<\gamma/2$), and so ${kn\choose x}\left(\frac{p}{1-p}\right)^x$ decreases on $[\ell,kn]$ (see, e.g.,~\cite[Chapter 1.2]{Bollobas}). Let us note that, for all $n$ large enough,
\begin{align*}
 {kn\choose\ell}={kn\choose sn} &=\frac{(1+o(1))\sqrt{k}}{\sqrt{2\pi s(k-s)n}}\cdot\frac{(kn)^{kn}}{(sn)^{sn}(kn-sn)^{n(k-s)}}\\
 & \leq\left(\frac{k^k}{s^s(k-s)^{k-s}}\right)^n \\
 &=\left((1+o(1))\frac{\sqrt{2\pi s(k-s)}}{\sqrt{k}}{k\choose s}\right)^n
  \leq\left({k\choose s}\ln n\right)^n
\end{align*}
implying that
$$
\rho\leq kn(\ln n)^n\biggl(\mathbb{P}[\mathrm{Bin}(k,p)=s]\biggr)^n.
$$ 
 
Thus, the desired expectation is upper bounded by
$$
 k (n\ln n)^n \biggl(\mathbb{P}[\mathrm{Bin}(k,p)=s]\biggr)^n\leq
 k\exp[-n\ln\ln n]=o(1).
$$
The last inequality follows from (\ref{eq:bin_upper}), completing the proof.


\section{Efficient search for large chordal graphs}
\label{sc:complexity}

Chordality of graphs can be efficiently tested, and a perfect elimination ordering of vertices of a chordal graph can be found in linear time~\cite{RTL}. Thus, if with bounded away from 0 probability it is possible to find in $G_n$ a chordal graph of size at least $(1+\varepsilon)n\log_{1/p}n$ in polynomial time, then it is also possible to find a clique of size $(1+\varepsilon)\log_{1/p}n$ in polynomial time, but the latter is an open problem that has received a lot of attention in the past (see, e.g.,~\cite{AKS, DGP, JP, Kucera}). We show that a chordal graph of size $(1+o(1))n\log_{1/p}n$ can be found in $G_n$ in polynomial time whp which is, as observed, asymptotically tight unless the mentioned large clique search problem can be efficiently solved.\\

Let $m=\lfloor \ln^2 n\rfloor$, $k=\lfloor\log_{1/p}n-4\log_{1/p}\ln n\rfloor$. For brevity, we call a path consisting of $m$ vertices, an {\it $m$-path}.

\begin{claim}
There is a polynomial-time algorithm that finds in $G_n$ a union of $\lfloor n/m\rfloor$ disjoint $k$-th powers of $m$-paths whp.
\end{claim}

\begin{proof}
We use an algorithm proposed by Alon and F\"{u}redi~\cite{AF}. Let $n'=\lfloor n/m\rfloor$. Consider any balanced partition $[mn']=V_1\sqcup\ldots\sqcup V_m$. We will order the vertices in every $V_i$ sequentially starting from $V_1$. The first set is ordered arbitrarily: $V_1=\{v^1_1,\ldots,v^1_{n'}\}$. Let $i\in\{2,\ldots,m\}$, and assume that, for every $i'\leq i-1,$ $V_{i'}=\{v^{i'}_1,\ldots,v^{i'}_{n'}\}$ is already ordered in such a way that, for every $j\in[n']$, the vertex $v^{i'}_{j}$ is adjacent to each of $v^{\max\{1,i'-k\}}_{j},\ldots,v^{i'-1}_{j}$. Consider an auxiliary bipartite graph $H_i$ with equal parts $V_i=\{v_1,\ldots,v_{n'}\}$ (the order is initially arbitrary), $U_i$, where the $j$-th vertex of $U_i$ is 
$$
\mathbf{u}^i_j=\left\{v^{\max\{1,i-k\}}_j,\ldots,v^{i-1}_j\right\}.
$$
We draw an edge between $v_{j'}$ and $\mathbf{u}^i_j$ if and only if $v_{j'}$ is adjacent to each of $v^{\max\{1,i-k\}}_j,\ldots,v^{i-1}_j$ in $G_n$. Note that, if we find a perfect matching $M_i$ (which is known to be solvable in polynomial time~\cite{MV}) in $H_i$, then we get the desired ordering of $V_i$ and, eventually, $\cup_i M_i$ corresponds to a union of $k$-th powers of $m$-paths in $G_n$ as needed.

Observe that $H_i$ is a binomial bipartite graph with edges appearing with probability at least $p^k\geq\frac{\ln^4 n}{n}$. Thus, with probability $1-o(1/n)$ it has a perfect matching (see, e.g.~\cite[Corollary 7.13]{Bollobas}). By the union bound, every $H_i$, $i\in[m]$, contains a perfect matching $M_i$ and we can order every $V_i$ in the desired way. \end{proof}

It remains to notice that a union of $\lfloor n/m\rfloor$ disjoint $k$-th powers of $m$-paths is a chordal graph with more than $\lfloor n/m\rfloor\left(m-k\right)k=kn(1-o(1))$ edges.

\section{Sparse random graphs: proof of Theorem~\ref{th:main_2}}
\label{sc:sparse}

For a sequence of graphs $G_n$ on $[n]$ and a fixed graph $F$, an {\it almost $F$-tiling} of $G_n$ is a sequence of subgraph $H_n\subset G_n$ on $n(1-o(1))$ vertices formed by disjoint unions of graphs isomorphic to $F$. We need to recall the result of Ruci\'{n}ski~\cite{Rucinski} about the threshold for the existence of almost tilings in the random graph. Let us recall that the {\it 1-density} of $F$ is $\rho(F)=\frac{|E(F)|}{|V(F)|-1}$. Let us call $F$ {\it strictly 1-balanced}, if every proper subgraph of $F$ has 1-density strictly less than $\rho(F)$. The {\it maximum 1-density} of a graph $F$ is the maximum 1-density over all its subgraphs: $\rho^*(F)=\max_{H\subseteq F}\rho(H)$.

\begin{thm}[Ruci\'{n}ski, 1992~\cite{Rucinski}]
Let $F$ be a graph with a maximum 1-density $\rho^*$ and $p\gg n^{-1/\rho^*}$. Then whp $G(n,p)$ has an almost $F$-tiling.
\label{th:Rucinski}
\end{thm}

Let $p=n^{-\alpha+o(1)}$, for a constant $\alpha>0$. We consider two cases from the statement of Theorem~\ref{th:main_2} separately: in Section~\ref{sc:alpha>1/2} we denote by $k:=k_{\alpha}$ the largest integer such that $\alpha<\frac{1+k}{1+2k}$ and prove the theorem for $\alpha\in(1/2,1]$; in Section~\ref{sc:alpha<1/2}, we address the case $\alpha\in(0,1/2]$.

\subsection{$\alpha>1/2$}
\label{sc:alpha>1/2}



If $\frac{1}{n}\ll p\ll n^{-2/3}$, then whp $G_n$ has a connected component of size $n(1-o(1))$ (see, e.g.,~\cite{Bollobas,Janson}) implying that whp $X_n\geq n-o(n)$ (a tree that covers almost all vertices of $G_n$ is chordal). On the other hand, the expected number of triangles is $o(n)$, and the expected number of 4-cliques is $o(1)$ implying that whp there are $o(n)$ triangles and no 4-cliques by Markov's inequality. Then, whp for any chordal graph $H\subset G_n$ and its perfect elimination ordering, no vertices have outdegree 3 and $o(n)$ vertices have outdegree 2. Thus, whp $X_n\leq n+o(n)$. We eventually get that $X_n/n\stackrel{\mathbb{P}}\to 1$.\\ 


Other $\alpha>\frac{1}{2}$ can be handled similarly: let $n^{-\frac{1+k}{1+2k}}\ll p \ll n^{-\frac{2+k}{3+2k}}$ for some integer $k\geq 1$.  
 Let $H$ be the square of a path 
  of length $k+1$. Consider a graph $F=F(M)$ obtained by gluing sequentially $M$ copies of $H$: the first vertex of the $i$-th copy is glued with the last vertex of the $(i-1)$-th copy. Note that $F$ is chordal and its maximum 1-density equals $\rho^*(F)=\rho^*(H)=\frac{1+2k}{1+k}$. By Theorem~\ref{th:Rucinski}, whp there exists an almost $F$-tiling in $G_n$. In other words, whp there is a disjoint union of copies of $F$ (which is chordal) that cover $n(1-o(1))$ vertices of $G_n$. Clearly, for every $\varepsilon>0$, there exists $M$ such that
an almost $F$-tiling has at least $(\frac{1+2k}{1+k}-\varepsilon)n$ edges for $n$ large enough. Thus, whp $X_n/n\geq\frac{1+2k}{1+k}-o(1)$.\\

It remains to show that whp $X_n/n\leq\frac{1+2k}{1+k}+o(1)$. 
For every positive integer $i$, let $\mathcal{H}_i$ be the family of all chordal graphs on $0\sqcup[i]$ that can be obtained in the following way:
\begin{itemize} 
\item $0$ and $1$ are adjacent,
\item for every $j\in[i-1]$, the vertex $j+1$ is adjacent to at least 2 vertices in $0\sqcup [j]$,
\item for every $j\in[i-1]$, the neighbours of $j+1$ in $0\sqcup [j]$, form a clique.
\end{itemize}
Note that each graph from $\mathcal{H}_i$ has $i+1$ vertices and at least $2i-1$ edges. Set $\mathcal{H}=\sqcup_{i\geq 1}\mathcal{H}_i$. 

\begin{claim}
Every 2-connected chordal graph $H$ is isomorphic to a graph from $\mathcal{H}$.
\label{cl:2-connected_chordal}
\end{claim}

\begin{proof}
Assume that $H$ does not have a copy in $\mathcal{H}$. Let $\prec$ be a perfect elimination ordering of $H$, and let $T$ be the respective perfect elimination tree of $H$ (see Section~\ref{sc:ub} for the definition of a perfect elimination tree). Without loss of generality, assume that $V(H)=\{0,1,\ldots,i\}$, and $0\prec 1\prec\ldots\prec i$. We then get that there exists a vertex $j\in\{2,\ldots,i\}$ such that its outdegree is at most 1. Let $\nu(j)$ be the parent of $j$ in $T$. Let us consider a subtree $T'$ of $T$ that is induced by $\nu(j)$, $j$, and all descendants of $j$. We get that there are no edges between $V(T)\setminus V(T')$ and $V(T')\setminus\{\nu(j)\}$ in $H$. Indeed, otherwise, let $j'$ be the minimum vertex from $V(T')\setminus\{\nu(j)\}$ that is adjacent to a vertex $u$ from $V(T)\setminus V(T')$. We have $j'\neq j$, and the parent $\nu(j')$ of $j'$ in $T'$ is smaller than $j'$ and should be adjacent to $u$ as well due to chordality --- a contradiction.
\end{proof}

Note that, for $i>i'$, if $H\in\mathcal{H}_i$, then $H[\leq i']\in \mathcal{H}_{i'}$. Recall that every graph from $\mathcal{H}_i$ has at least $2i-1$ edges. Thus, by Markov's inequality, there exists $M=M(k)$ such that whp $G_n$ does not have a copy of any graph from $\mathcal{H}_M$, and, therefore, whp $G_n$ does not have a copy of any graph from $\mathcal{H}_{\geq M}$. Let $\mathcal{H}'\subset\sqcup_{i=1}^{k+1}\mathcal{H}_i$ be the set of all graphs $H$ from $\sqcup_{i=1}^{k+1}\mathcal{H}_i$ that have at least $2|V(H)|$ edges (i.e., $H\in\mathcal{H}_i$ belongs to $\mathcal{H}'$ if and only if there exists $j\leq i$ sending at least 3 edges to $0\sqcup [j-1]$ in $H$). Let $\mathcal{H}''=\sqcup_{i=k+2}^{M-1}\mathcal{H}_i$. We then observe that, for every graph $H\in\mathcal{H}'\sqcup\mathcal{H}''$, the expected number of subgraphs isomorphic to $H$ in $G_n$ is $o(n)$: a graph $H\in\mathcal{H}'$ on $i+1$ vertices has at least $2i$ edges, thus the expected number of copies of $H$ in $G_n$ is $O(n^{i+1}n^{-2i\alpha})=O(n^{1+i(1-2\alpha)})=o(n)$. In the same way, a graph $H\in\mathcal{H}''$ on $i+1$ vertices has at least $2i-1$ edges, thus the expected number of copies of $H$ in $G_n$ is $O(n^{i+1}n^{-(2i-1)\alpha})=O(n^{1+(k+2)-(2k+3)\alpha})=o(n)$. By Markov's inequality, whp there are $o(n)$ subgraphs isomorphic to a graph from $H\in\mathcal{H}'\sqcup\mathcal{H}''$ in $G_n$.

Now, let us consider an arbitrary chordal subgraph $F\subset G_n$, and let $\Sigma$ be the set of all its non-empty blocks (i.e. inclusion-maximal 2-connected subgraphs and edges that do not belong to any cycle, see~\cite[Chapter 3.1]{Diestel}). Without loss of generality, we may assume that $F$ is connected (in particular, there are no empty blocks) since whp $G_n$ is connected, and so the same is true for any its inclusion-maximal chordal subgraph. Note that every edge of $F$ belongs to some graph in $\Sigma$. By Claim~\ref{cl:2-connected_chordal}, we get that whp all graphs from $\Sigma$ have copies in $\sqcup_{i=1}^{M-1}\mathcal{H}_i$. Moreover, whp the total number of edges in the graphs from $\Sigma$ that have copies in $\mathcal{H}'\sqcup\mathcal{H}''$ is $o(n)$. For every $H\in\Sigma$ that has a copy in $\sqcup_{i=1}^{k+1}\mathcal{H}_i\setminus\mathcal{H}'$, we get 
that $H$ has at most $k+2$ vertices, and $|E(H_i)|=2|V(H_i)|-3$. Since the block-cutpoint graph is a tree, we get that there is an ordering of graphs from $\Sigma=\{H_1,\ldots,H_K\}$ such that, for every $i$, $H_i$ has at most 1 common vertex with $H_1\cup\ldots\cup H_{i-1}$. So, $|E(H_i)\setminus E(H_1\cup\ldots\cup H_{i-1})|=|E(H_i)|$ while $|V(H_i)\setminus V(H_1\cup\ldots\cup H_{i-1})|\geq|V(H_i)|-1=i$. Letting $x_H$ be the number of graphs from $\Sigma$ isomorphic to $H$, we get that whp
$$
\rho(F)\leq\frac{o(n)+\sum_{i=1}^{k+1}\sum_{H\in\mathcal{H}_i\setminus\mathcal{H}'}(2i-1)x_H}{\sum_{i=1}^{k+1}\sum_{H\in\mathcal{H}_i\setminus\mathcal{H}'}ix_H}\leq\frac{1+2k}{1+k}+o(1).
$$
The latter inequality hods true since $F$ is connected implying that the denominator in the left-hand side fraction is linear in $n$. Thus, the number of edges in $F$ is at most $\frac{1+2k}{1+k}n+o(n)$, completing the proof.

\subsection{$\alpha\leq 1/2$}
\label{sc:alpha<1/2}



 We first prove the upper bound: let us fix $\varepsilon>0$ and prove that whp $X_n\leq (1/\alpha+\varepsilon)n$. Since whp $G_n$ does not have $(\lceil2/\alpha\rceil+1)$-cliques, whp in a chordal subgraph of $G_n$ every vertex has outdegree at most $\lceil2/\alpha\rceil$. Moreover, arguing similarly to the proof of the upper bound of Theorem~\ref{th:main} from Section~\ref{sc:ub}, we see that the expected number of chordal subgraphs of $G_n$ with at least $(1/\alpha+\varepsilon)n$ edges and outdegrees at most $\lceil2/\alpha\rceil$ (for some perfect elimination ordering) can be upper bounded by 
$$
n^{n-1}A^n p^{(1/\alpha+\varepsilon)n}\leq (An^{1-1-\alpha\varepsilon+o(1)})^n=o(1)
$$
for some constant $A>0$. Here, $n^{n-1}$ is the number of rooted trees that can play a role of a perfect elimination tree and $A^n$ is the bound for the number of ways to choose vectors $\mathbf{e}_i$ (see Section~\ref{sc:ub}). Thus, by Markov's inequality, we get that whp $X_n/n\leq 1/\alpha+o(1)$ as needed.\\

We then prove the lower bound. Let us first assume that $1/\alpha=\ell\geq 2$ is an integer. For every $j\in\mathbb{N}$, consider the following chordal graph $F_j$ on $\{0\}\sqcup [j]$: the vertex $i\in[j]$ is adjacent to its $\min\{j-1,\ell\}$ predecessors, i.e. $F_j$ is the $\ell$-th power of a path of length $\ell$. Note that $F_j$ is a strictly 1-balanced graph with 1-density strictly less than $\ell$ and approaching $\ell$ with growing $j$. By Theorem~\ref{th:Rucinski}, for every $j$, whp $G_n$ has an almost $F_j$-tiling, implying that $X_n/n\geq(1/\alpha-o(1))$ whp.

Now, let $1/\alpha\notin\mathbb{Z}$. Let $\ell$ be the minimum positive integer greater than $1/\alpha$. Note that $\ell\geq 3$. Let us consider the (unique) sequence $x_i$, $i\in\mathbb{N}$, satisfying the following conditions:
\begin{itemize}
\item $(x_1,\ldots,x_{\ell})=(1,2,\ldots,\ell)$;
\item for every $i\geq\ell+1$, we define recursively $x_i$ to be the maximum integer in $[2,\ell]$ such that 
$$
\rho_i:=\frac{x_1+\ldots+x_i}{i}<\frac{1}{\alpha}.
$$
\end{itemize}
Let us consider the sequence $(s_j)_{j\in\mathbb{N}}$ of all $s$ such that
\begin{itemize}
\item $\rho_s>\ell-1$,
\item $\rho_s>\rho_i$ for all $i<s$.
\end{itemize} 
Since $\rho_i\uparrow 1/\alpha>\ell-1$, there are infinitely many such $s$. For every $j$, set $\mathbf{x}_j=(x_1,\ldots,x_{s_j})$. Consider the following graph $F_j$ on $\{0\}\sqcup [s_j]$: the vertex $i\in[s_j]$ is adjacent to its $x_i$ immediate predecessors. Note that $\rho_{s_j}$ is the 1-density of $F_j$. Let us prove that $F_j$ is chordal. It is sufficient to show that the natural order of integers is perfect elimination. We first observe that, for every $i>\ell$, we have $x_i\in\{\ell-1,\ell\}$. Indeed, the inequality $x_1+\ldots+x_{i-1}<\frac{1}{\alpha}(i-1)$, that holds for all $i>\ell$, implies $x_1+\ldots+x_{i-1}+\ell-1<\frac{1}{\alpha}i$ since $\ell-1\leq\frac{1}{\alpha}$. Also, the first $\ell+1$ vertices of $F_j$ compose a clique. Proceeding by induction, we get that any set of $\ell$ consecutive vertices in $F_j$ induces a clique. Thus, every $i\geq\ell+1$ is adjacent to at most its $\ell$ predecessors, and these predecessors induce a clique. So the considered order is indeed perfect elimination.

Since the 1-density of $F_j$ approaches $1/\alpha$ as $j$ grows, then, for every $\varepsilon>0$ and large enough $j$, an almost $F_j$-tiling has at least $(1/\alpha-\varepsilon)n$ edges. Therefore, we conclude that $X_n/n\geq(1/\alpha-o(1))$ by applying Theorem~\ref{th:Rucinski} which is possible since all $F_j$ are strictly 1-balanced due to the following claim. This completes the proof of Theorem~\ref{th:main_2}.

\begin{claim}
All graphs $F_j$ are strictly 1-balanced.
\end{claim}
\begin{proof}

Fix $j$ and assume that $F_j$ is not strictly 1-balanced. Let $\tilde F$ be a proper subgraph of $F_j$ that has 1-density $\rho(\tilde F)\geq \rho_{s_j}$. First of all let us note that without loss of generality $\tilde F$ is connected since the 1-density of $\tilde F$ does not exceed 1-densities of all its connected components. Also, we may assume that $\tilde F$ is an induced subgraph. 

Let us assume that in $\tilde F$ some {\it intermediate} vertices are missing, i.e. there exist $i_-,i,i_+\in V(F_j)$ such that $i_-<i<i_+$ and $i_-,i_+\in V(\tilde F)$, while $i\notin V(\tilde F)$. Note that the number of consecutive missing vertices could not be bigger than $\ell-1$ since otherwise $\tilde F$ is not connected. Let $i\geq 1$ be the minimum number such that, for some $\mu\in[\ell-1]$,
\begin{center}
 $i,i+\mu+1\in V(\tilde F),\,\,$ while $\,i+1,\ldots,i+\mu\notin V(\tilde F)$.
\end{center} 
If $i\geq\ell-1$, then the missing vertices from $[i+1,i+\mu]$ add at least $\mu\ell>\mu/\alpha$ edges to $\tilde F$. Indeed, every missing vertex sends at least $\ell-1$ edges to its immediate predecessors in $F_j$, and is adjacent to $i+\mu+1$. Thus, the inequality $\rho(F_j)\leq \rho(\tilde F)$ implies $\rho(F_j)<\rho(\tilde F')$, where the graph $\tilde F'$ is obtained from $\tilde F$ by adding back the vertices $i+1,\ldots,i+\mu$. If $i<\ell-1$, then $i+1$ vertices $i'\leq i$ contribute at most $(\ell-1)(i+1)$ edges to $E(\tilde F)$. Indeed, every vertex $i'\leq i$ is adjacent to at most $\ell$ and at least $\ell-1$ its immediate successors in $F_j$. Since the missing vertex $i+1$ is adjacent to all $i'\leq i$ in $F_j$, we get that every $i'\leq i$ is adjacent to at most $\ell-1$ its successors in $\tilde F$. Thus, the deletion of vertices $i'\leq i$ from $\tilde F$ leads to the graph $\tilde F'$ with $\rho(\tilde F')>\rho(F_j)$. We conclude that there is a proper subgraph in $F_j$ with the 1-density at least $\rho(F_j)$ and without missing intermediate vertices. Thus, without loss of generality, we may assume that $\tilde F$ is induced by $[i_-,i_+]$ for certain $0< i_-<i_+\leq s_j$. Note that the first inequality is strict since otherwise we get a contradiction with the definition of $F_j$. 


Let $\nu>\ell$ be the minimum number such that $x_{\nu}=\ell-1$. We have that, among the first $\nu-1$ elements of the sequence $\mathbf{x}_j$, there are exactly $\nu-\ell$ numbers equal to $\ell$. Assume first that $\nu-\ell+1\leq i_-\leq\nu-1$. Then the missing vertices $0,1,\ldots,i_--1$ add at least $(\ell-1)(i_--(\nu-\ell))+\ell(\nu-\ell)$ edges to $\tilde F$. Note that $\frac{(\ell-1)(i_--(\nu-\ell))+\ell(\nu-\ell)}{i_-}$ is minimised at $i_-=\nu-1$:
$$
\frac{(\ell-1)(i_--(\nu-\ell))+\ell(\nu-\ell)}{i_-}\geq
\frac{(\ell-1)^2+\ell(\nu-\ell)}{\nu-1}.
$$
We also recall that 
$$
 \frac{|E(F_j[\leq \nu-1])|+\ell}{(|V(F_j[\leq \nu-1])|-1)+1}=\frac{\ell(\ell-1)/2+(\nu-\ell+1)\ell}{\nu}\geq\frac{1}{\alpha},
$$
implying that 
\begin{align*}
 (\ell-1)^2+\ell(\nu-\ell) & =\nu\ell-2\ell+1\geq
 \frac{\nu}{\alpha}+\frac{\ell^2-\ell}{2}-2\ell+1\\
 &> \frac{1}{\alpha}(\nu-1)+(\ell-1)+\frac{\ell^2-\ell}{2}-2\ell+1\\
 &=\frac{\nu-1}{\alpha}+\frac{\ell^2-3\ell}{2}\geq\frac{\nu-1}{\alpha},
\end{align*}
since $\frac{1}{\alpha}>\ell-1$ and $\ell\geq 3$. So, the ``additional density'' recovered by the vertices $0,1,\ldots,i_--1$ equals 
$$
\frac{(\ell-1)(i_--(\nu-\ell))+\ell(\nu-\ell)}{i_-}>\frac{1}{\alpha}.
$$
If $i_-\leq\nu-\ell$, then the missing vertices $0,1,\ldots,i_--1$ add at least $i_-\ell$ edges to $\tilde F$. Thus, in both cases, the addition of vertices $0,1,\ldots,i_--1$ leads to the graph $\tilde F'$ satisfying the inequality $\rho(\tilde F')>\rho(F_j)$ --- contradiction with the property of $F_j$ that all $F_j[\leq s]$, $s<s_j$, have smaller 1-densities. 

It remains to consider the case $i_-\geq\nu$. Clearly, we may assume that 
 $\tilde F$ has at least $\ell+1$ vertices since otherwise $\rho(\tilde F)\leq\frac{\ell}{2}<\ell-1<\rho_{s_j}$ due to the fact that $\ell\geq 3$ and the choice of $s_j$. Due to the definition of $\mathbf{x}_j$, we have that, for every $i\geq\nu$,
$$
 \frac{i}{\alpha}-1\leq x_1+\ldots+x_i<\frac{i}{\alpha}.
$$
Let $\delta_i:=\frac{i}{\alpha}-(x_1+\ldots+x_i)$.
We then get
\begin{align}
 \frac{i_+-i_-}{\alpha}-(x_{i_-+1}+\ldots+x_{i_+}) &=
 \frac{i_+}{\alpha}-(x_1+\ldots+x_{i_+})-
 \left(\frac{i_-}{\alpha}-(x_1+\ldots+x_{i_-})\right)\notag\\
 &\geq \delta_{i_+}-1.
 \label{eq:intermediate_density}
\end{align}
On the other hand, the vertex $x_{i_-+1}$ sends $1\leq(x_{i_-+1}-(\ell-2))$ edge to $x_{i_-}$ in $\tilde F$, the vertex $x_{i_-+2}$ sends $2\leq(x_{i_-+2}-(\ell-3))$ edges to $\{x_{i_-},x_{i_+}\}$  in $\tilde F$, etc, implying 
$$
\rho(\tilde F)  \leq\frac{x_{i_-+1}+\ldots+x_{i_+}-(1+\ldots+\ell-2)}{i_+-i_-}
\leq \frac{x_{i_-+1}+\ldots+x_{i_+}-1}{i_+-i_-}.
$$ 
Therefore, due to~\eqref{eq:intermediate_density}, we get
\begin{align*}
\rho(\tilde F) & \leq\frac{x_{i_-+1}+\ldots+x_{i_+}-1}{i_+-i_-}\\
&\leq
\frac{(i_+-i_-)/\alpha-\delta_{i_+}+1-1}{i_+-i_-}\\
&=\frac{1}{\alpha}-\frac{\delta_{i_+}}{i_+-i_-}<
\frac{1}{\alpha}-\frac{\delta_{i_+}}{i_+}\\
&=\rho(F_j[\{0,1,\ldots,i_+\}])\leq\rho(F_j)
\end{align*}
by the definition of $F_j$ --- a contradiction.
\end{proof}

\section{Concluding remarks}
\label{sc:remarks}

In this paper, we study maximum chordal subgraphs in random graphs. We have found asymptotics of maximum sizes of chordal subgraphs in dense and sparse random graphs.

We believe that the concentration interval in Theorem~\ref{th:main} is not optimal, and, in particular there exists $c\in\mathbb{R}$ and $C>0$ such that whp 
$$
\left|X_n-\gamma n\log_{1/p} n-cn\log_{1/p}\log n\right|\leq Cn.
$$
Unfortunately, our techniques does not seem sufficient even to achieve the constant factor in the second-order term. Also, we are inclined to believe that $X_n$ is not concentrated in any interval of length $o(n)$.

Note that whp the maximum size of a chordal subgraph of $G(n,c/n)$, $c>0$, is $n-Y_n+o(n)$, where $Y_n$ is the number of connected components since whp $G(n,c/n)$ has $o(n)$ triangles. Thus $X_n/n\stackrel{\mathbb{P}}\to 1-\gamma(c)\in(0,1)$, where $\gamma(c)$ is the limit in probability of $Y_n/n$, which is well known~\cite[Theorem 6.12]{Bollobas}:
$$
 \gamma(c)=\frac{1}{c}\sum_{i=1}^{\infty}\frac{i^{i-2}}{i!}(c e^{-c})^i.
$$
In the same way, we believe that, for every $k\in\mathbb{N}$ and $\alpha=\frac{1+k}{1+2k}$, letting $p=cn^{-\alpha}$, we would get 
\begin{equation}
X_n/n\stackrel{\mathbb{P}}\to\gamma_k(c)\in\left(\frac{2k-1}{k},\frac{2k+1}{k+1}\right),
\label{eq:conj_sparse}
\end{equation}
where $\gamma_k(c)$ increases in $c$, and $\lim_{c\to 0}\gamma_k(c)=\frac{2k-1}{k}$, $\lim_{c\to\infty}\gamma_k(c)=\frac{2k+1}{k+1}$. A possible approach for proving (\ref{eq:conj_sparse}) is to analyse the behaviour of inclusion-maximal subgraphs in $G(n,p)$ consisting of blocks of chordal graphs with the outdegree sequence $0,1,2,\ldots,2$ of length $k+3$. 

It would be interesting to study maximum sizes of subgraphs of the random graphs that belong to other families of perfect graphs (interval graphs, strongly chordal graphs, co-graphs, etc). Note that the maximum size of a perfect graph in $G(n,p=\mathrm{const})$ equals $(1/4+o(1))pn^2$ whp, and the same is true for any hereditary family that contains all bipartite graphs but does not contain at least one 3-colourable graph~\cite{AKSamotij}.


\begin{thebibliography}{10}

\bibitem{AF} N. Alon, Z. F\"{u}redi, {\it Spanning subgraphs of random graphs}, Graphs and Combinatorics, {\bf 8} (1992) 91--94.

\bibitem{AKSamotij} N. Alon, M. Krivelevich, W. Samotij, {\it Largest subgraph from a hereditary property in a random graph}, Discrete Mathematics, {\bf 346}:9 (2023) 113480.

\bibitem{AKS} N. Alon, M. Krivelevich, B. Sudakov, {\it Finding a large hidden clique in a random graph}, Random Structures \& Algorithms, {\bf 13} (1998) 457--466.

\bibitem{Bollobas} B. Bollob\'{a}s, {\bf Random Graphs}, second edition, Cambridge University Press, 2001.

\bibitem{BF} B. Bollob\'{a}s, A. Frieze, {\it Spanning Maximal Planar Subgraphs of Random Graphs}, Random Structures \& Algorithms, {\bf 2} (1991) 225--231.

\bibitem{BPS} G. Brightwell, K. Panagiotou, A. Steger, {\it Extremal subgraphs of random graphs}, Random Structures \& Algorithms, {\bf 41} (2012) 147--178.

\bibitem{Buneman} P. Buneman, {\it A characterisation of rigid circuit graphs}, Discrete Mathematics, {\bf 9}:3 (1974) 205--212.

\bibitem{CMS} A. Coja-Oghlan, C. Moore, V. Sanwalani, {\it Max $k$-cut and approximating the chromatic number of random graphs}, Random Structures \& Algorithms, {\bf 28}:3 (2006) 289--322.

\bibitem{CG} D. Conlon, W. T. Gowers, {\it Combinatorial theorems in sparse random sets}, Annals of Mathematics, {\bf 184} (2016) 367--454. 

\bibitem{DGP} Y. Dekel, O. Gurel-Gurevich, Y. Peres, {\it Finding hidden cliques in linear time with high probability}, Combinatorics, Probability \& Computing, {\bf 23}:1 (2014) 29--49. 

\bibitem{MK} B. DeMarco, J. Kahn, {\it Mantel's theorem for random graphs}, Random Structures \& Algorithms, {\bf 47} (2015) 59--72.

\bibitem{DMS} A. Dembo, A. Montanari, S. Sen, {\it Extremal cuts of sparse random graphs}, Annals of Probability, {\bf 45}:2 (2017) 1190--1217.

\bibitem{Diestel} R. Diestel, {\bf Graph Theory}, 5th edition, Springer, 2017.

\bibitem{chordal} G. A. Dirac, {\it On rigid circuit graphs}, Abhandlungen Mathematischen Seminar Universit\"{a}t Hamburg, {\bf 25} (1961) 71--76.



\bibitem{EL} P. Erd\H{o}s, R. Laskar, {\it A note on the size of a chordal subgraph}, Southeastern International Conference on Graphs, Combinatorics and Computing, Utilitas Mathematica Publishing Co (1985) 81--86.


\bibitem{GL} D. Gamarnik, Q. Li, {\it On the max-cut of sparse random graphs}, Random Structures \& Algorithms, {\bf 52}:2 (2018) 219--262.

\bibitem{Computers} M. R. Garey, D. S. Johnson, {\bf Computers and Intractability: A Guide to the Theory of NP-Completeness}, W. H. Freeman \& Co., 1990.

\bibitem{Gish} L. Gishboliner, talk at the workshop ``{\it Recent advances in probabilistic and extremal combinatorics}'', Ascona, August 2022.

\bibitem{GS} L. Gishboliner, B. Sudakov, {\it Maximal chordal subgraphs}, Combinatorics, Probability \& Computing, {\bf 32}:5 (2023) 724--741. 

\bibitem{Golumbic_book} M. C. Golumbic, {\bf Algorithmic Graph Theory and Perfect Graphs}, second edition, Elsevier, 2004.

\bibitem{BGW_book} M. C. Golumbic, {\it Chordal graphs}. In {\bf Topics in Algorithmic Graph Theory} (Encyclopedia of Mathematics and its Applications) pp. 130--151, edited by L.~W.~Beineke, M.~C. Golumbic, R.~J.~Wilson, Cambridge University Press, 2021. 

\bibitem{HS} I. Hoshen, W. Samotij, {\it Simonovits's theorem in random graphs} (2023) arXiv:2308.13455.

\bibitem{Janson} S. Janson, T. \L uczak, A. Ruci\'{n}ski, {\bf Random Graphs}, J. Wiley \& Sons, 2000.

\bibitem{JP} A. Juels, M. Peinado, {\it Hiding cliques for cryptographic security}, Designs, Codes and Cryptography, {\bf 20}:3 (2000) 269--280.

\bibitem{KKM} M. Krivelevich, G. Kronenberg, A. Mond, {\it Tur\'{a}n-type problems for long cycles in random and pseudo-random graphs},
Journal of the London Mathematical Society, {\bf 107} (2023) 1519--1551.

\bibitem{Kucera} L. Ku\v{c}era, {\it A generalised encryption scheme based on random graphs}, Graph-Theoretic Concepts in Computer Science, 18th International Workshop (WG 1992), Lecture Notes in Computer Science, {\bf 570} (1991) 180--186.

\bibitem{MV} S. Micali, V. V. Vazirani, {\it An $O(\sqrt{|V|}|E|)$ algorithm for finding maximum matching in general graphs}, 21st Annual Symposium on Foundations of Computer Science (FOCS 1980), Syracuse, NY, USA (1980) 17--27. 

\bibitem{RTL} D. J. Rose, R. E. Tarjan, G. S. Lueker, {\it Algorithmic aspects of vertex elimination on graphs}, SIAM Journal on Computing, {\bf 5}:2 (1976) 266--283.

\bibitem{Rucinski} A. Ruci\'{n}ski, {\it Matching and covering the vertices of a random graph by copies of a given graph}, Discrete Mathematics, {\bf 105} (1992) 185--197.

	
\bibitem{Schacht} M. Schacht, {\it Extremal results for random discrete structures}, Annals of Mathematics, {\bf 184}:2 (2016) 333--365.

\bibitem{VA_survey} L. Vandenberghe, M. S. Andersen, {\bf Chordal Graphs and Semidefinite Optimization}, Foundations and Trends in Optimization, {\bf 1}:4 (2015) 241--433.

\bibitem{West} D. B. West, {\bf Introduction to Graph Theory}, second edition,  Prentice Hall, 2001.



\end{thebibliography}
\end{document}